\documentclass[letterpaper, 10 pt, conference]{ieeeconf}  

\IEEEoverridecommandlockouts                              
\title{\LARGE \bf Tutorial on Congestion Control in Multi-Area Transmission Grids \\via Online Feedback Equilibrium Seeking
}

	\author{Giuseppe Belgioioso, Saverio Bolognani, Giulia Pejrani, and Florian D\"orfler
		\thanks{This work was supported by NCCR Automation, a National Centre of Competence in Research, funded by the Swiss National Science Foundation (grant number 180545).}
\thanks{The authors are with Automatic Control Laboratory, Department of Electrical Engineering and Information Technology,
        ETH Zurich, Physikstrasse 3 8092 Zurich, Switzerland. (e-mail:{\tt \{gbelgioioso, bsaverio, dorfler\}@ethz.ch}, \tt{gpejrani@student.ethz.ch}).
        }
		}

\usepackage[utf8]{inputenc}
\usepackage[T1]{fontenc}
\usepackage{graphicx}
\usepackage{xcolor}
\usepackage{amsmath}
\usepackage{amssymb}
\usepackage{amsthm}
\usepackage{amssymb}
\usepackage{caption}
\usepackage{subcaption}
\usepackage{hyperref}
\usepackage{cite}
\usepackage[gen]{eurosym}
\usepackage{balance}
\DeclareMathOperator{\subjectto}{subject~to}

\newtheorem{theorem}{Theorem}
\newtheorem{definition}{Definition}

\newtheorem{remark}{Remark}
\newtheorem{assumption}{Assumption}

\newcommand{\col}{\textrm{col}}

\newcommand{\proj}{\mathrm{proj}}

\newcommand{\Id}{\mathrm{Id}}

\newcommand{\diag}{\operatorname{diag}}

\newcommand{\nc}{\mathcal{N}}

\newcommand{\reals}{\mathbb{R}}

\newcommand{\minimize}{\textrm{minimize}}

\newcommand{\mto}[0]{\rightrightarrows}
\newcommand{\eigmax}[1]{{\lambda}_{\max} \left({#1}\right)}
\newcommand{\eigmin}[1]{{\lambda}_{\min}\left({#1}\right)}

\newcommand{\mc}{\mathcal}
\newcommand{\bb}{\mathbb}
\newcommand{\R}{\bb R}

\begin{document}

\maketitle
\thispagestyle{empty}
\pagestyle{empty}

\begin{abstract}
Online feedback optimization (OFO) is an emerging control methodology for real-time optimal steady-state control of complex dynamical systems. This tutorial focuses on the application of OFO for the autonomous operation of large-scale transmission grids, with a specific goal of minimizing renewable generation curtailment and losses while satisfying voltage and current limits. When this control methodology is applied to multi-area transmission grids, where each area independently manages its congestion while being dynamically interconnected with the rest of the grid, a non-cooperative game arises. In this context, OFO must be interpreted as an online feedback equilibrium seeking (FES) scheme. Our analysis incorporates technical tools from game theory and monotone operator theory to evaluate the stability and performance of multi-area grid operation. Through numerical simulations, we illustrate the key challenge of this non-cooperative setting: on the one hand, independent multi-area decisions are suboptimal compared to a centralized control scheme; on the other hand, some areas are heavily penalized by the centralized decision, which may discourage participation in the coordination mechanism.
\end{abstract}

\section{Introduction}

To achieve climate goals and enhance energy independence, an increasing amount of renewable power is being integrated into the grid \cite{World2021, Renewable2021} and is replacing traditional form of generation.
This renewable generation is typically dispersed across the grid or located where the primary energy source (wind, solar) is mostly available.
Moreover, generation from renewable sources follows temporal patterns that are not fully predictable and depend as well on the generation technology.
These temporal and spatial variability of generation, together with the increasing demand for electricity (also driven by the electrification of road transportation) poses unprecedented challenges on the operation of power transmission grids: as the network reaches its limits, renewable power often needs to be curtailed to avoid congestion issues such as overloaded lines and over-voltages. 

Manual or semi-automated mechanisms for congestion control in power transmission grids are unsuited for these new tasks, for two main reasons:
\begin{itemize}
    \item with sampling times in the order of minutes, they cannot for safely respond to generation variability;
    \item they do not scale to the complexity presented by the huge number of small-size renewable generators.
\end{itemize}
Consequently, there is a growing need to enhance real-time automation and implement control actions at shorter intervals, ideally every few seconds. The French operator RTE estimates savings of billions of Euros over a decade if real-time control of power flows can prevent the construction or reinforcement of power lines \cite{Schema2019}. In the US, the connection of more that 2000 GW of renewable generation is currently being delayed because of grid capacity constraints \cite{osti_1969977}.

Online feedback optimization (OFO) has been proposed as an effective approach to design real-time congestion control mechanisms for the electric power grid \cite[Section IV]{molzahn2017survey}. OFO lies at the intersection of feedback control design and nonlinear optimization: its goal is to design a control policy that makes the optimal steady state of the system (i.e., the uncongested operation of the grid) an attractive equilibrium point for the closed-loop dynamics. 
This approach has been studied in more general settings (not limited to power grids) in a number of recent papers \cite{HauswirthDoerfler21c, bernstein2019online, lawrence2020linear, Colombino2020, Bianchin2022, simonetto2020time}. 

In the domain of power distribution grids, these methods have been proposed in multiple variations (see the review in \cite[\S IV.D]{molzahn2017survey}), specialized to Volt/VAr regulation \cite{Bolognani2015,li2022robust, qu2019optimal} or more general optimality criteria \cite{DallAnese2018}, tested experimentally \cite{OrtmannBolognani20}, and even deployed in the real world \cite{kroposki2020good,OrtmannDoerfler23}.

The application to congestion control in the transmission grid has been explored less. For example, \cite{tang2020measurement} focuses on the voltage control problem, while \cite{OrtmannBolognani22} considers both voltage and power flow control, and introduces a novel sub-transmission benchmark model that represents a real French grid. We will review the main findings of this last work in the numerical experiments of Section~\ref{sec:singlearea}.

In these works, the transmission grid is a single ``plant'' and the objective of the proposed OFO controllers is to drive it to an operating point that minimizes a single global cost function.
In reality, however, modern transmission grids consist of different interconnected areas that are locally managed by different transmission systems operators (TSOs) but physically interconnected. This is the case, for example, of national and sub-national networks in continental Europe and of Independent System Operators in North America.
In such multi-area setting, a global OFO controller is not practical, as such scheme would requires a coordination and exchange of sensitive information between the areas, and may be not incentive-compatible. Local OFO controllers, however, would be readily deployable on the different areas, since their implementation only requires local measurements, the ability of controlling local generators, and a rough estimate of the static model of the local part of the grid. A natural question in this more realistic scenario is whether we can expect the overall multi-area transmission grid, resulting from the interconnection of the locally-controlled areas, to retain any stability and efficiency property. 

Game theory, which studies dynamics of conflict and cooperation between self-interested rational decision makers, offers a powerful framework to approach this question. The development of game-theoretic controllers for economic steady-state regulation of complex systems has been recently studied in different works \cite{BelgioiosoDoerfler21, BelgioiosoDoerfler22, agarwal2022game, romano2022game}. All these works focus on the design of game-theoretic feedback controllers, but the technical tools therein can be also used for analysis purposes.

In this tutorial, we present a game-theoretic analysis of autonomous real-time control of multi-area transmission grids, in which each area selfishly regulates its congestion using online feedback optimization while being dynamically interconnected with the rest of the grid. In the first part, we show how to design various feedback controllers for a single area of the transmission grid using the OFO framework. Further, we present realistic simulations based on the newly proposed benchmark in \cite{OrtmannBolognani22} which demonstrate their effectiveness in regulating the grid to optimal and safe operating points. In the second part, we investigate the behaviour of a multi-area transmission grid, where the sub-areas independently manage congestion using OFO while remaining dynamically connected through the transmission grid. We use tools from game theory and monotone operators to establish conditions for stability of the closed-loop dynamics of the grid. Finally, we present some numerical simulations on the on the IEEE 30 bus power flow test case to corroborate the theory and study the performance of the multi-area OFO control design.

We conclude the tutorial by discussing how multi-area congestion control in a competitive setting yields some intrinsic inefficiency, and how this defines a future research endeavour: the design of suitable incentives that ensure that the decision of multiple independent grid operators remains aligned with global efficiency and sustainability goals.

\subsection{Notation}
Given a positive definite matrix $P\!=\!P^\top$, $\|x\|_P = \sqrt{x^T P x}$. The largest (smallest) eigenvalue of $P$ is denoted by $\eigmax{P}$, ($\eigmin{P}$). Given $N$ vectors $x_1,\ldots,x_N$, we denote by $\col(x_1,\ldots,x_N)=[x_1^\top \ldots x_N^\top]^{\top}$ their vertical concatenation. Given $N$ scalar $y_1,\ldots,y_N$, we denote by $\diag(y_1,\ldots,y_N)$ the diagonal matrix with  $y_1,\ldots,y_N$ on the main diagonal. Given a closed convex set $\Omega \subseteq \reals^n$, $\iota_{\Omega}:\reals^n \to \{0,\infty\}$ denotes its indicator function, $\nc_{ \Omega}(x):\Omega \rightarrow \reals^n$ denotes its normal cone operator, and $\proj_{\Omega}:\reals^n \to \Omega$ is the Euclidean projection onto $\Omega$. A set-valued mapping $\mc B:\reals^n \mto \reals^n$ is monotone if it satisfies $(u-v)^\top (x-y) \geq 0 $ for all vectors $x, y \in \R^n$, $u \in \mathcal{B} (x)$, and $v \in \mathcal{B} (y)$. A single-valued mapping $F:\reals^n \rightarrow \reals^n$ is cocoercive, with parameter $\mu>0$, if it satisfies $(F(x)-F(y))^\top (x-y) \geq \mu \|F(u)-F(y)\|^2 $ for all vectors $x, y \in \R^n$.

\section{Single-Area Transmission Grid Control}
\label{sec:singlearea}
\label{eq:subtr}

The problem of real-time control of a transmission grid using online feedback optimization (OFO), which is schematically represented in Figure~\ref{fig:blockdiagram}, is characterized by
\begin{itemize}
    \item a set of \emph{inputs} that can be actuated, e.g., the curtailment of active power generation from renewable sources, voltage set-points for voltage-controlled generators, and reactive power set-points of inverter-based generators;
    \item a set of \emph{outputs} that can be measured on the system, including both line and bus measurements; for simplicity, we will assume that the full state of the grid (or a dynamic estimation, see \cite{PicalloDoerfler20}) is available.
\end{itemize}
In the remainder of the paper, we denote by $u$ this multidimensional input and by $y$ this multidimensional output.

These signals are related to each other by the algebraic model (i.e., the input-out model of the transmission grid)
\begin{equation}
\label{eq:hmap}
y=h(u; w),
\end{equation}
where $w$ is a set of exogenous disturbances, e.g., uncontrollable nodal power injections (demand and generation).

\begin{figure}[h]
    \centering
    \includegraphics[width=\columnwidth]{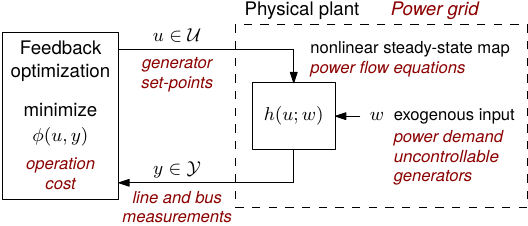}
    \caption{The feedback structure of Online Feedback Optimization of a single-area transmission grid.}
    \label{fig:blockdiagram}
\end{figure}

The function $h$ represents an input-output representation of the power flow equations.
These are usually given in implicit form, as they are often derived from Kirchhoff's and Ohm's laws and from algebraic models for loads and generators. In their implicit form, they determine a manifold of power flow solutions which are compatible with the physics of the system (see \cite{Bolognani_manifoldlinearization} for details on this interpretation of the power flow equations).
By assuming the existence of an input-output map $h$, we are also assuming that a chart for this manifold exists, at least in the area of interest, and that this local coordinate map is uniformly not degenerate in the input coordinates.
The technical requirement can be traced back to standard implicit-function-theorem conditions, which effectively corresponds to assuming voltage stability of the grid (invertibility of the power flow Jacobian \cite[Ch. 7]{Cutsem2007VoltageStability}).

An alternative, but compatible, interpretation is that $y=h(u; w)$ is the steady-state map of a stable dynamical system that represents the dynamics of power lines, generators, and loads. The analysis of online feedback optimization in the presence of plant dynamics goes beyond the scope of this tutorial (see for example \cite{HauswirthDoerfler21, Colombino2020, LawrenceLTI} and the other approaches reviewed in \cite[Section 4.2]{HauswirthDoerfler21c}).

\begin{remark}
The disturbance $w$ also contains the power flows from/to different neighboring parts of the grid. This modeling choice automatically assumes that the behavior of the rest of the grid is \emph{exogenous}, i.e., it does not respond to the decisions taken by the controller that we are designing.
We will see in Section~\ref{sec:multiarea} that this assumption is critical and needs to be revisited in the case of a multi-area setting.
\end{remark}

Finally, specifications for the real-time congestion control problem are given in the form of:
\begin{itemize}
    \item A set $\mathcal U:= \{u ~|~ Au \le b\}$ of \emph{feasible control inputs}, corresponding for example to the feasible region of the different actuators;
    \item A set $\mathcal Y:= \{y ~|~ Cy \le d\}$ of \emph{feasible outputs}, encoding for example voltage and line current limits;
    \item A continuously differentiable \emph{cost function} $\varphi(u,y)$ that represents the resulting operating cost of the grid; for example, the cost of active power curtailment of renewable generators, power losses, etc.
\end{itemize}

{
The resulting static optimization problem reads as
\begin{subequations}\label{eq:optimizationproblem}
\begin{align}
    \underset{u,y}{\minimize} & \quad \varphi(u,y) \\
    \subjectto & \quad y = h(u; w)  \label{eq:h-constraint}\\
    & \quad u \in \mathcal U \\
    & \quad y \in \mathcal Y. \label{eq:output-constraints}
\end{align}
\end{subequations}
In general, this problem is non-convex due to the equality constraint \eqref{eq:h-constraint}.
In the remainder of the paper, we work under the assumption that the problem is feasible, namely, a valid power flow solution exists.

Following the principle of online feedback optimization, we aim at designing an iterative update law that solves \eqref{eq:optimizationproblem}, i.e., that drives the decision variables $u,y$ to a local solution of \eqref{eq:optimizationproblem}.
}
To be compatible with the implementation in closed loop with the physical system, the algebraic constraint \eqref{eq:h-constraint} needs to be satisfied at all iterations. 
While this requirement reduces the degrees of freedom in the design, it is also a crucial source of robustness of the OFO approach: the availability of $y$ as a measurement allows to ``outsource'' the evaluation of the function $h$ to the physical system, reducing the reliance on model information in a substantial way.

Notice that this online setting reduces the design problem to the choice of the update iteration for the controlled decision variable $u$, while the decision variable $y$ is completely and automatically determined by \eqref{eq:h-constraint}. This setup is to be interpreted as the interconnection of the input iteration and the algebraic plant, and the design goal is to ensure that the local solution of \eqref{eq:optimizationproblem} are asymptotically stable equilibria for the closed-loop systems. See \cite{HauswirthDoerfler21b} for a formal discussion of the connection between local convexity properties of the problem and stability properties of the closed-loop interconnection, although for continuous-time dynamics.

Multiple iterations (mostly inspired by iterative algorithms in nonlinear optimization) serve this purpose, including
\begin{itemize}
    \item (projected) gradient iterations \cite{HauswirthDoerfler21c,HaeberleDoerfler21}
    \item primal-dual saddle-point dynamics \cite{HauswirthDoerfler21c,Colombino2020,Bernstein2019}
    \item safe gradient flows \cite{allibhoy2022control}
    \item regularized primal-dual iterations \cite{Bianchin2022}
    \item quasi-Newton flows \cite{Low2017}
    \item sequential convex programming \cite{BelgioiosoDoerfler22}
    \item and others.
\end{itemize}

We refer to \cite{BelgioiosoDoerfler22}, for a general framework encompassing most of the algorithms listed above.
An important difference between these iterations is how they handle the constraints of the problem, and in particular the output constraints \eqref{eq:output-constraints}.
To illustrate this point, we quickly review a few alternatives.

\paragraph{Penalty function and gradient descent iteration}

A penalty function $p(y)$ can be used as a proxy for the output constraint \eqref{eq:output-constraints}, yielding the approximate problem
\begin{subequations}
\label{eq:penalty}
\begin{align}
       \underset{u,y}{\minimize} & \quad \overbrace{\varphi(u,y) + p(y)}^{:= J(u,y)} \label{eq:pen_ref_constr2}\\
    \subjectto & \quad y = h(u; w) \label{eq:pen_ref_constr1}\\
    & \quad u \in \mathcal U.
\end{align}
\end{subequations}
When the penalty term $p$ is chosen continuously differentiable, after substituting \eqref{eq:pen_ref_constr1} within \eqref{eq:pen_ref_constr2}, a projected gradient descent iteration for this problem takes the form
%
\begin{multline}
    \label{eq:gradient-descent}
    u^{k+1} = \textrm{proj}_{\mathcal{U}}\big[
        u^{k} - \alpha \nabla_u J(u^k,y^k) \\
        - \alpha \nabla_u h(u^k;w)^\top  
            \nabla_y J (u^k,y^k)
    \big],
\end{multline}
where $\alpha$ is a tunable gain/step size.
Convergence of this iteration to the set of local minimizers can be guaranteed without convexity assumption (see \cite{HauswirthDoerfler21c} and \cite{BelgioiosoDoerfler22} for a continuous-time and a sampled-data stability analyses, respectively).

\paragraph{Projected gradient descent iteration}

Alternatively, we can maintain the original output constraint $y \in \mathcal Y$ in \eqref{eq:optimizationproblem} and perform a projected steepest-descent iteration. 
Following the derivation in~\cite{HaeberleDoerfler21}, the resulting update for the decision variable $u$ takes the form 
\begin{equation} \label{eq:feedbackupdate} 
    u^{k+1} = u^k + \alpha \sigma(u^k,w,y^k)
\end{equation}
where
\begin{align}\label{eq:projection_QP}
\begin{split}
    \sigma(u,w,y) =\\
    \arg\min_{\delta u} \  &\| \delta u 
    + \nabla_u \varphi (u,y) + \nabla_u h(u;w)^\top \nabla_y \varphi (u,y)  \|^2
    \\ 
    \subjectto \  &A (u+\delta u)\leq b \\ & C (y+\nabla_u h(u;w) \delta u)\leq c.
\end{split}
\end{align}
Also in this case, convergence to the set of local minima can be guaranteed without convexity assumptions: Theorem 3 in \cite{HaeberleDoerfler21} states that under weak technical assumptions on the problem and with sufficiently small (but not vanishing) gain $\alpha$, the closed loop system is guaranteed to converge to the set of first-order optimal points of~\eqref{eq:optimizationproblem}, and aysmptotically stable equilibria are strict local minima.

\paragraph{Primal-dual iteration}

By dualizing the output constraints and by substituting $y=h(u;w)$, we obtain a Lagrangian for \eqref{eq:optimizationproblem} that takes the form
\[
\mathcal L(u,\lambda) 
= \varphi(u,h(u;w)) + \lambda^\top \left( Cy - d \right).
\]
A saddle points of $\mathcal L$ can be determined by the primal-dual iteration
\begin{align*}
    u^{k+1} &= \textrm{proj}_{\mathcal{U}}\big[
        u^{k} - \alpha \nabla_u \varphi(u^k,y^k) \\
        & \quad  - \alpha \nabla_u h(u^k;w)^\top 
        \left( 
            \nabla_y \varphi (u^k,y^k) + C^\top \lambda^k
        \right)
    \big]\\
    \lambda^{k+1} &= \textrm{proj}_{\mathbb R_{\ge 0}}\big[
            \lambda^k + \beta \left( Cy^k - d \right)
        \big].
\end{align*}

The conditions for convergence are not equally weak in this case: convexity conditions are needed both to connect saddle-points of the Lagrangian to the minima of the original problem, and to prove convergence of the iteration \cite{HauswirthDoerfler21c}.

\smallskip 
A common feature of all these iterations is that the only piece of model information that is needed is $\nabla_u h(u,w)$, namely, the sensitivity of the output $y$ with respect to the input $u$.
In the application at hand, this corresponds to the sensitivity of the solution of the state of the power grid with respect to the controllable set-points of the generators.
From a modeling perspective, these sensitivities are directly related to the first-order linearization of the power flow equations, i.e., the local coordinate chart of the power flow manifold (see \cite{Bolognani_manifoldlinearization} for a closed-form expression).
In practice, this is a well-known quantity in power flow analysis and a generalization of the concept of Power Transfer Distribution Factors. These sensitivities are often available to the operators as byproduct of their grid state-estimation and planning.

We stress that none of these iterations require to evaluate the plant model $y=h(u;w)$ numerically, as they rely on directed measurements of $y$, which dramatically increase the robustness of these approaches to model uncertainty, as observed both in numerical simulations \cite{OrtmannBolognani22} and in real-world experiments \cite{OrtmannBolognani20,OrtmannDoerfler23}.
Some formal guarantees of such robustness have been derived, based on different models of the uncertainty in \cite{Colombino2019robustness,Pasquini2021robustness}.

\subsection{Numerical experiments: Subtransmission grid control}

\begin{figure}[t]
    \centering
    \includegraphics[width=\columnwidth]{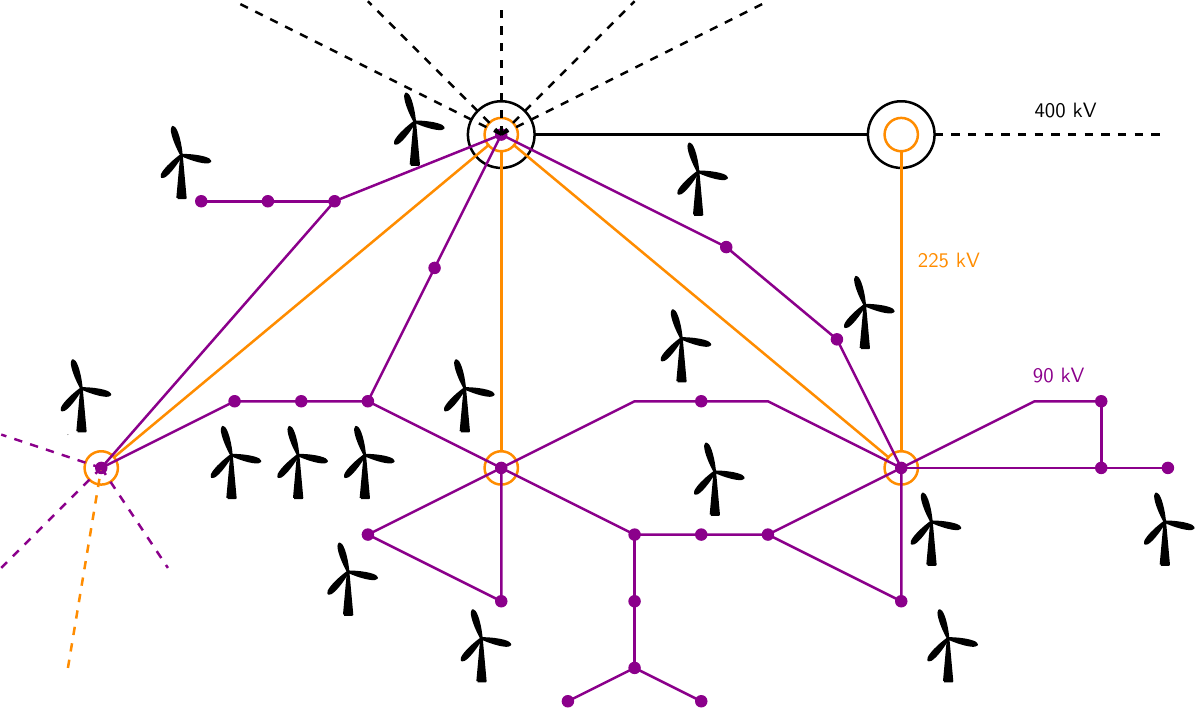}
    \caption{Schematic of the transmission grid in the area of Blocaux, France (31 buses, 58 branches), which constituted the benchmark for OFO controllers in  \cite{OrtmannBolognani22}. The dashed lines represent connections to neighboring parts of the grid.}
    \label{fig:blocaux}
\end{figure}

We quickly review the results in \cite{OrtmannBolognani22},
where the OFO strategy was tested via numerical experiements on a benchmark model of the theal French subtransmission grid. The entire grid consists of 7019 buses, 9657 branches, and 1465 generators.
The task in the benchmark is to minimize the losses and active power curtailment in the Blocaux area (schematically represented in Figure~\ref{fig:blocaux}). The area hosts 42 wind farms with power ratings between 0.5~MW and 102~MW, and a total installed wind power of 1274~MW.
The real-time controller decides active power curtailment and reactive power injection of the individual wind farms.\footnote{The controller tested in \cite{OrtmannBolognani22} is also capable to control the position of tap changer transformers, a discrete input that we excluded from this paper.}
While doing so, the controller was required to satisfy voltage magnitude limits at the buses and power flow limits on the lines.
It is assumed that the rest of the generators in the grid (both inside and outside the Blocaux area) maintain their set-point constants in the meanwhile.

A projected gradient descent controller as that in \eqref{eq:feedbackupdate} has been adopted for its ability to guarantee steady-state satisfaction of the constraints and better transient performances.

\begin{figure}[t]
    \centering
    \includegraphics[width=0.89\columnwidth]{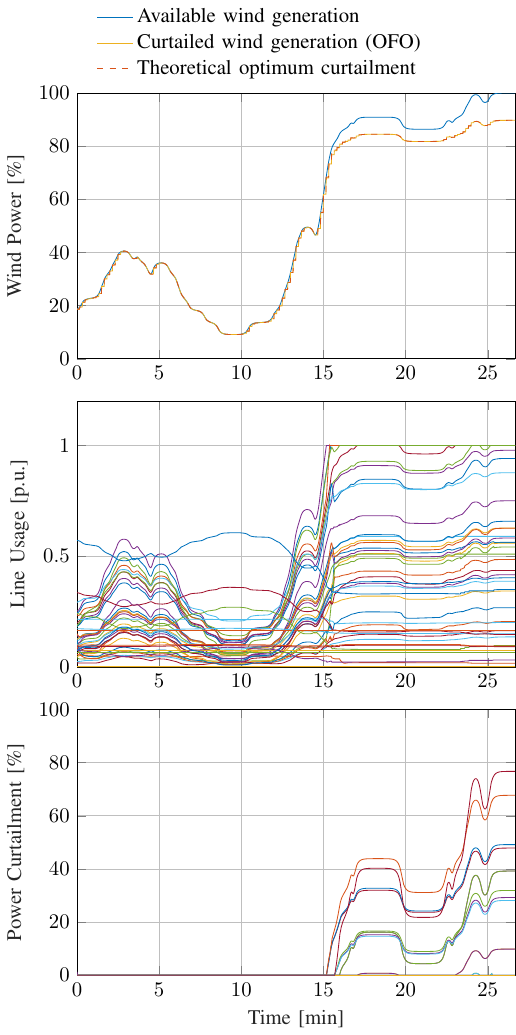}
    \caption{Performance of the feedback-optimization controller on the single-area UNICORN benchmark in the Blocaux area in France. We refer to the full simulation study in \cite{OrtmannBolognani22} for the details, which include the reactive power set-points and the resulting regulation of bus voltages to the desired range.
    }
    \label{fig:centralized}
\end{figure}

Exemplary results of a simulation can be seen in Figure~\ref{fig:centralized}, where the response to a fast variation of available wind power is reported.
It can be seen that the controller successfully tracks the optimal curtailment (which is a combination of multiple small curtail actions at each individual wind generator).
Note that the curtailment action is only used when line congestion constraints become active, and successfully prevents the overload of the lines.
We refer to \cite{OrtmannBolognani22} for a complete analysis of the performance of OFO on this benchmark transmission grid and for a discussion on the tuning of the proposed real-time control law.

\section{Multi-area Trasmission Grid Control}
\label{sec:multiarea}

In this section, we investigate the behaviour of multi-area transmission grids in which the areas selfishly regulate their congestion using OFO controllers, while being  dynamically interconnected with each other, as shown in Figure~\ref{fig:MAEx}.

We consider a transmission grid consisting of $N$ interconnected areas, labelled by $i\in \mathcal{I}:= \{1,\ldots,N\}$. Each area $i\in \mathcal{I}$ controls its local inputs $u_i$, namely, active and reactive power setpoints of local generators, and measures its local outputs $y_i$, namely, voltage and current on the local lines. We assume that current magnitudes of lines that interconnect two areas are output measurements for both areas. With a game-theoretic notation, we use $u = \col(u_1,\ldots, u_N)$ to denote the stacked vector of control inputs of all areas, $u_{-i} = \col(u_1,\ldots,u_{i-1},u_{i-1}, \ldots,u_N)$ to denote the stacked vector of control inputs of all areas but area $i$, and $(u_i,u_{-i}) = u$.

In this multi-area setting, the state of each area $i$ depends not only on the local control $u_i$, but also on the control inputs of the other areas, $u_{-i}$, since the areas are dynamically interconnected via the grid. Formally, the control inputs $(u_i, u_{-i})$ and the local output $y_i$ is related by mapping
\begin{equation}
\label{IOSS_MA}
y_i = h_i(u;w) = h_i((u_i,u_{-i});w).
\end{equation}

In practice, each area $i\in \mathcal{I}$ aims at selfishly minimizing local operational costs (i.e., renewable generation curtailment of local generators, losses, etc.) while ensuring voltage and current safety limits on the local lines, yielding
\begin{align}
\label{eq:game}
\forall i \in \mathcal{I}:
\quad
\min_{u_i \in \mathcal{U}_i} \; \big\{ J_i(u_i, y_{i}) \;|\; y_i = h_i((u_i,{u_{-i}}); w)\big\},
\end{align}
where we used the same penalty-based formulation described in \eqref{eq:penalty}, in which the output constraints $y_i \in \mathcal{Y}_i$ are replaced with a penalty functions $p_i$ that is part of the cost functions $J_i$, i.e., $J_i(u_i,y_i) = \varphi_i(u_i,y_i) + p_i(y_i)$.

Note that these optimization problems are inter-dependent, i.e., the optimal control input of area $i$ depends on the control inputs of the other areas, and thus constitute a noncooperative game. Here, we assume that each area $i\in \mathcal{I}$ approaches its correspondent congestion control problem problem using the OFO controller in \eqref{eq:gradient-descent}.
This yields the dynamic feedback law
\begin{align}
\label{eq:OFOsinglearea}
u_i^{k+1} = \textrm{proj}_{\mathcal{U}_i}\left[
u_i^{k}- \gamma_i F_i(u^k,y_i^k)
\right],
\end{align}
where the mapping $F_i$ is defined as
\begin{align}
\label{eq:locGrad}
F_i(u,y_i) &= \nabla_{u_i} J_i(u_i,y_i) + \nabla_{u_i} h_i(u;w)^\top \nabla_{y_i} J_i(u_i,y_i).
\end{align}
Intuitively, $F_i$ is the partial gradient of the cost function $J_i(u_i,h((u_i,u_{-i});w))$ with respect to $u_i$, and $\nabla_{u_i} h_i((u_i,u_{-i});w)$ is the Jacobian of the local steady-state mapping \eqref{IOSS_MA} also with respect to the local decision variable $u_i$ only. The latter corresponds to the sensitivity of the solution of the state of the local area of the grid with respect to the controllable set-points of the local generators.
%
%

\begin{figure}[t]
    \centering
    \includegraphics[width=\columnwidth]{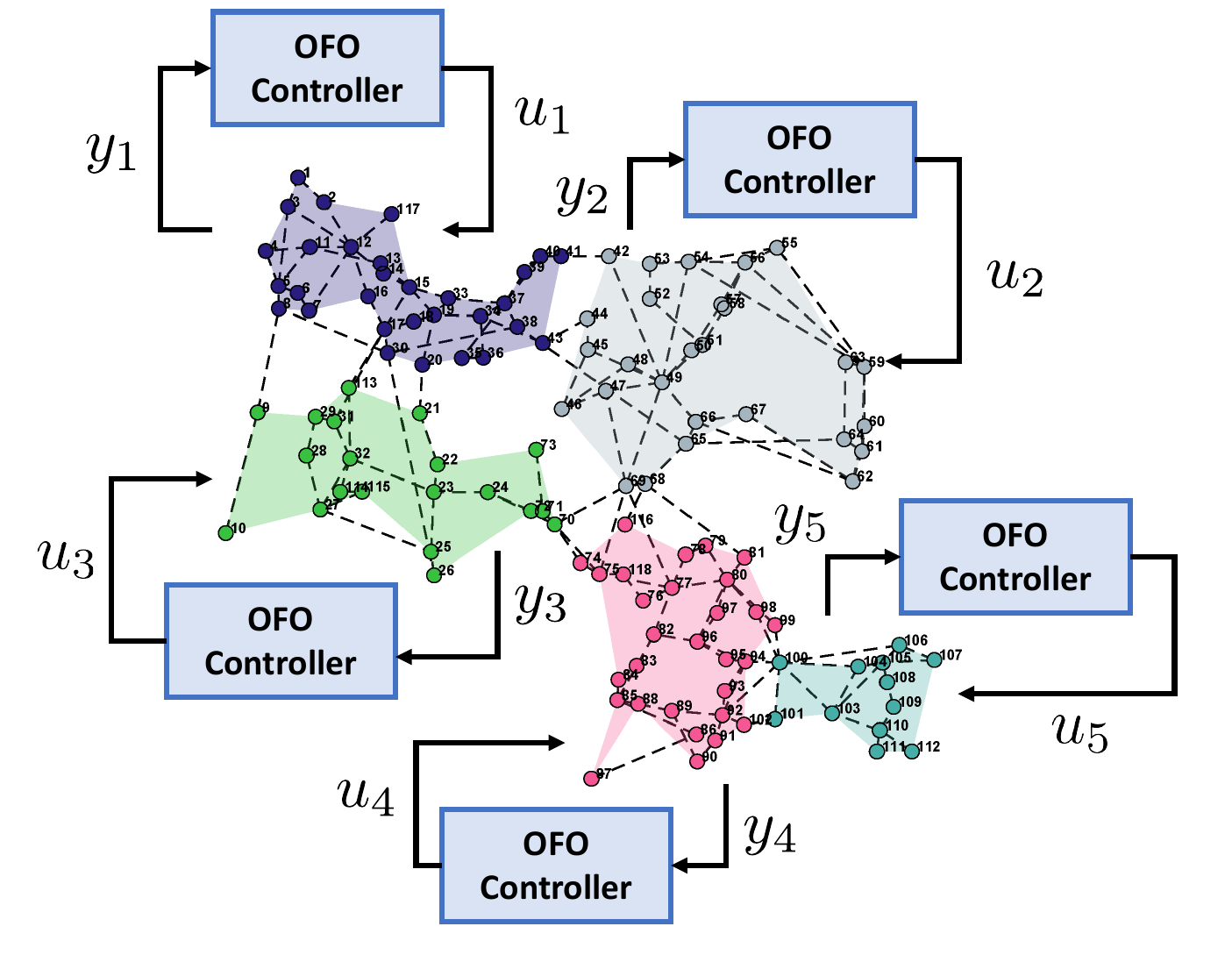}
    \caption{An example of a 5-area partition over the IEEE 118 bus test case (from \cite{hossain2022multi}). Each area has access to voltage and current measurements on the local lines, and regulates its congestions using Online Feedback Optimization.}
    \label{fig:MAEx}
\end{figure}

A natural question in this multi-area setting is whether the dynamics of the closed-loop transmission grid-OFO controllers interconnection is stable or not.
In the rest of this section, we show that, under some regularity conditions on the primitives of the game \eqref{eq:game} and appropriate design choices for the gains $\gamma_i$, the multi-area OFO controllers \eqref{eq:OFOsinglearea} indeeed stabilize the transmission grid to a competitive equilibrium.

\begin{definition}
A feasible control profile $u^*$ is a Nash equilibrium of the game in \eqref{eq:game} if, for all $i \in \mathcal{I}$,
\begin{align}
J_i(u_i^*,h(u^*;w)) \leq J_i(u_i,h((u_i,u_i^*);w)), \quad \forall u_i \in \mathcal{U}_i.
\end{align}
\end{definition}

Before presenting the main stability result, we introduce some preliminary technical assumptions.
\begin{assumption} \label{ass:PGmon} For all $i \in \mathcal I$, $J_i(u_i,h(u;w))$ is convex with respect to $u_i$ and continuously differentiable. The pseudo-gradient $\mathbb{F}$ of the game \eqref{eq:game}, defined as $\mathbb{F}(u) := \col(F_1(u_1,h_1(u;w)),\ldots,F_N(u_N,h_N(u;w)))$, is $\mu$-cocoercive, i.e., $\forall u,u' \in \mathcal{U}$
\begin{align}
\langle \mathbb{F}(u)-\mathbb{F}(u'), \, u-u'\rangle \geq \mu \|\mathbb{F}(u)-\mathbb{F}(u')\|^2.
\end{align}
\end{assumption}
This assumption is quite standard in the literature of game theory \cite{belgioioso2022distributed} and is one of the weakest under which existence of a Nash equilibrium can be proven.

\begin{assumption} \label{ass:StepSize}
For all $i \in \mathcal{I}$, $\gamma_i \in (0,2\mu)$.
\end{assumption}
This assumption on the local gains $\gamma_i$ limits the aggressiveness of the local OFO controllers and, consequently, its tracking performance in a time-varying setting.
With these assumptions in place, we are ready to prove stability.
\begin{theorem} \label{th:main}
Under Assumptions \ref{ass:PGmon} and \ref{ass:StepSize}, the sequence of output $\{y^k\}_{k \in \mathbb N}$ generated by \eqref{eq:OFOsinglearea} converges to some $y^*=h(u^*;w)$, where $u^*$ is a Nash equilibrium of the game \eqref{eq:game}.
\end{theorem}
\begin{proof}
The proof is given in the appendix.
\end{proof}

\subsection{Simulations on the IEEE 30 bus power flow test case}

In this section, we study the properties of the multi-area OFO controllers in \eqref{eq:OFOsinglearea} via numerical simulations on the IEEE 30 bus power flow test case. This benchmark grid consists of $41$ lines and $30$ buses with the following elements: 5 PV generators, 3 PQ generators, 20 loads, and 1 external grid (slack bus).
Further, we assume that the grid is divided in three areas $\mathcal I :=\{1,2,3\}$, locally controlled but physically coupled. An illustration of the composition, topology, and partition of the grid is shown in Figure~\ref{fig:IEEE30bus}.
In all simulations, we used Pandapower \cite{thurner2018pandapower}, an AC power flow solver, to compute the output $y$ of the power flow, i.e., $h(u;w)$. 
\begin{figure}[t]
    \centering
    \includegraphics[width=\columnwidth]{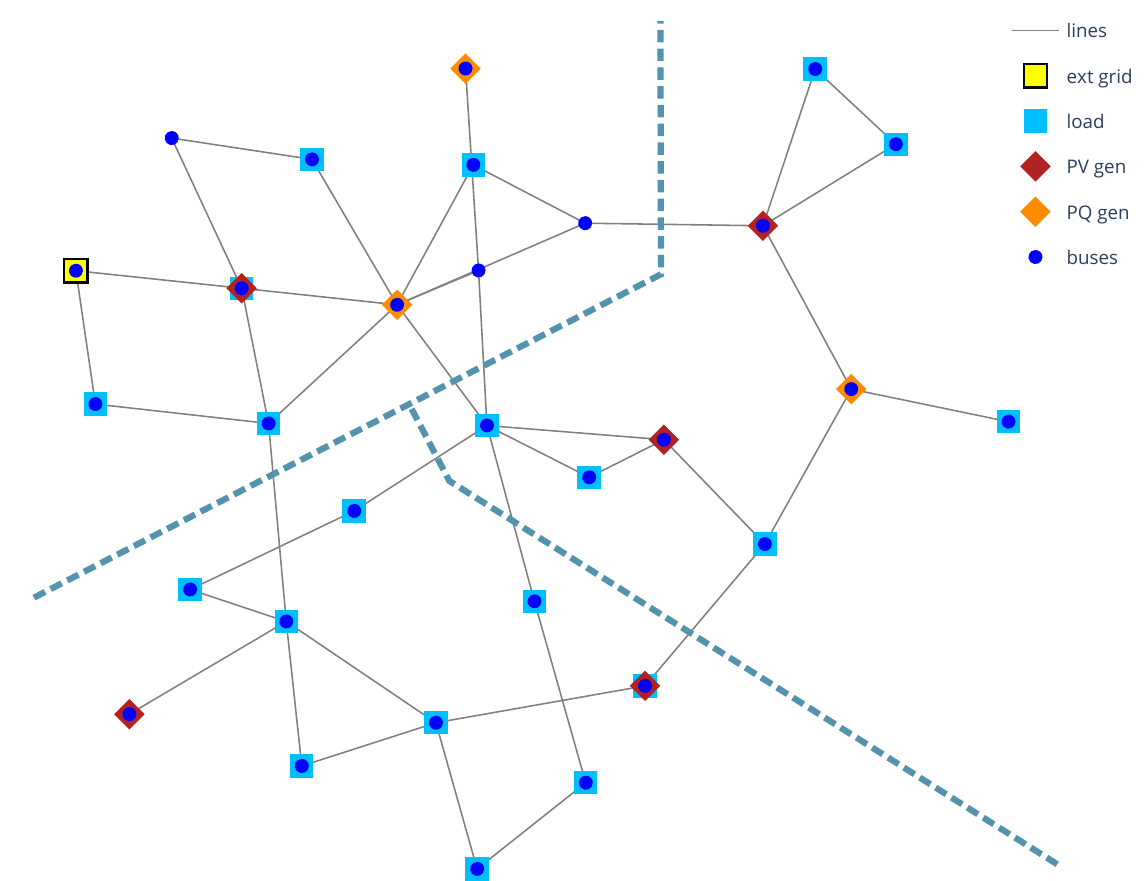}
    \caption{Composition and topology of the IEEE 30 bus power flow test case. The dashed lines represent the division of the grid into three areas.}
    \label{fig:IEEE30bus}
\end{figure}

\subsubsection{Stability of multi-area OFO control}
In this first case study, we show that the multi-area OFO controllers in \eqref{eq:OFOsinglearea} indeed stabilize the transmission grid to a competitive steady-state equilibrium. We simulate 134 minutes of closed-loop grid operation, with a 10 seconds delay between sensing and actuation. Namely, every 10 seconds, the controllers receive field measurements $y_i$ of the grid state, computes the local gradients as in \eqref{eq:locGrad}, and update the set points of the controllable generators $u_i$ for the next iteration.  Figures~\ref{fig:output1} and \ref{fig:output2} show the resulting evolution of the grid current and voltage magnitudes, respectively. We see that after a quick transitory, both voltage and magnitude stabilize to a steady-state, which corresponds to a Nash equilibrium of the game \eqref{eq:game}, by Theorem~\ref{th:main}. We also note that such competitive steady-state satisfies the output constraints. In particular, we observed that the multi-area OFO controllers quickly drive the current of line 9, which is initialized in an overloaded state, within the safe limits, as shown Figures~\ref{fig:output1}.

\begin{figure}[t]
    \centering
    \includegraphics[width=\columnwidth]{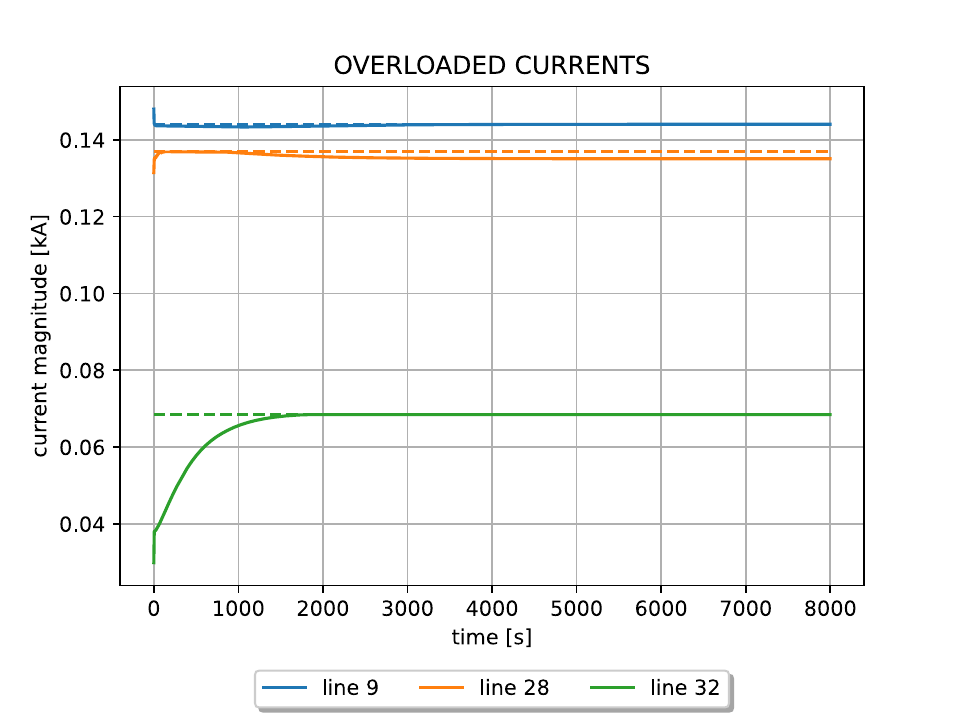}
    \caption{Evolution of the line currents (solid lines) of the IEEE 30 bus grid in Figure~\ref{fig:IEEE30bus} under the multi-area OFO controllers \eqref{eq:OFOsinglearea}. The dashed lines represents the correspondent line limits.}
    \label{fig:output1}
\end{figure}
\begin{figure}[t]
    \centering
    \includegraphics[width=\columnwidth]{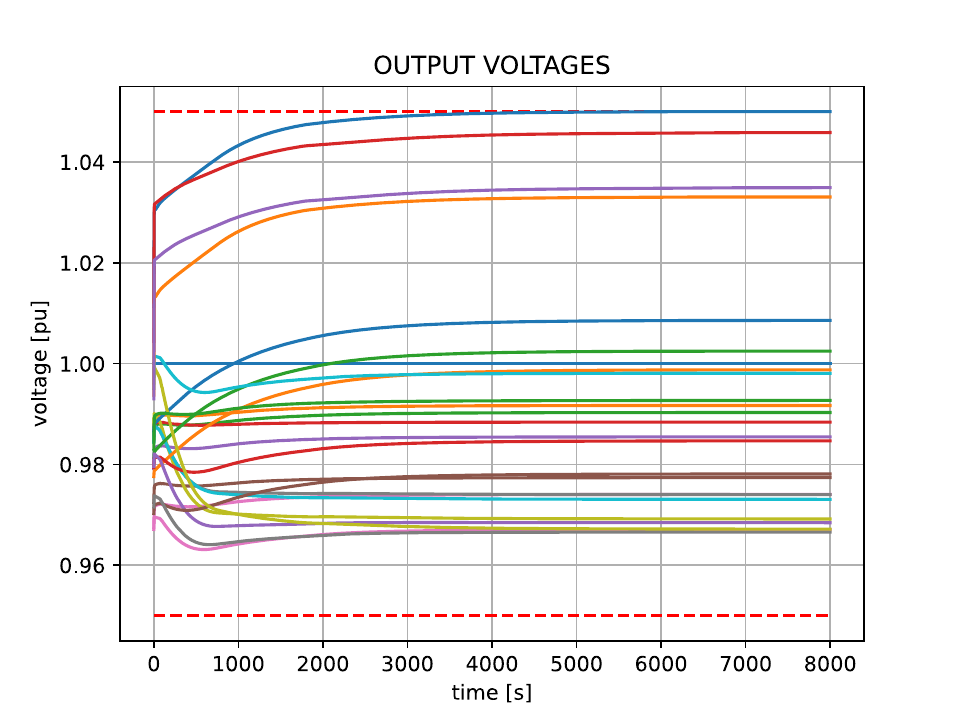}
    \caption{Evolution of the line voltages of the IEEE 30 bus grid in Figure~\ref{fig:IEEE30bus} under the multi-area OFO controllers \eqref{eq:OFOsinglearea}.}
    \label{fig:output2}
\end{figure}

\subsubsection{Multi-area OFO versus centralized OFO control}
In this case study, we compare the operation and the economic performance of the multi-area OFO controllers in \eqref{eq:OFOsinglearea} against the centralized OFO controller obtained by applying the iteration \eqref{eq:gradient-descent} to the grid as a single area. Under appropriate choices of the control gains $\alpha$ and $\gamma_i$'s, both control designs are able to stabilize the grid to feasible operating points, as predicted by the theory. However, while the centralized OFO controller stabilizes the grid to a social optimal state, multi-area OFO control stabilizes the grid to a competitive (Nash) state. The resulting total active power curtailment and the correspondent curtailment cost for each area at the two different operating points are plotted in Figure~\ref{fig:tables_cropped}. Note that while the total active power curtailed is comparable ($\sim$138 MW in both cases), the cumulative curtailment cost of the multi-area case is 81\% larger than the centralized case. This is not surprising as the centralized OFO controller is specifically designed to minimize the social cost at steady-state. On the other, the resulting steady-state is not strategically stable, namely, some areas may unilaterally decrease their local costs by changing their control inputs.

\begin{figure}[t]
    \centering
    \includegraphics[width=\columnwidth]{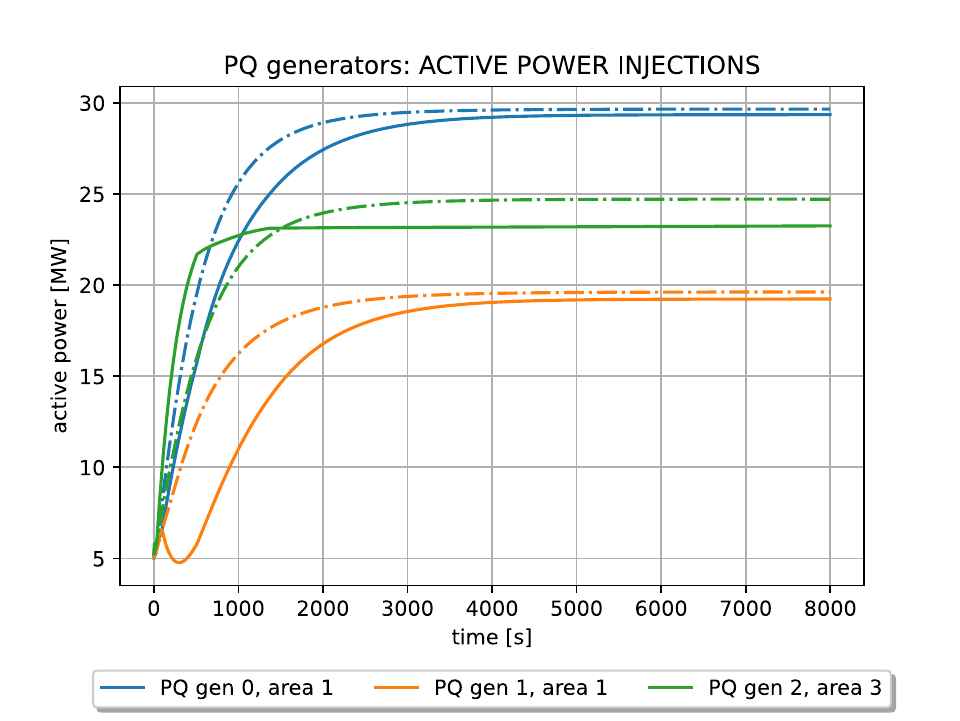}
    \caption{Evolution of the active power injections of the PQ generators in the different areas of the grid in Fig.~\ref{fig:IEEE30bus} under the multi-area OFO controllers (solid lines) and the centralized OFO controller (dot-dashed lines), respectively.}
    \label{fig:PQ}
\end{figure}
\begin{figure}[t]
    \centering
    \includegraphics[width=\columnwidth]{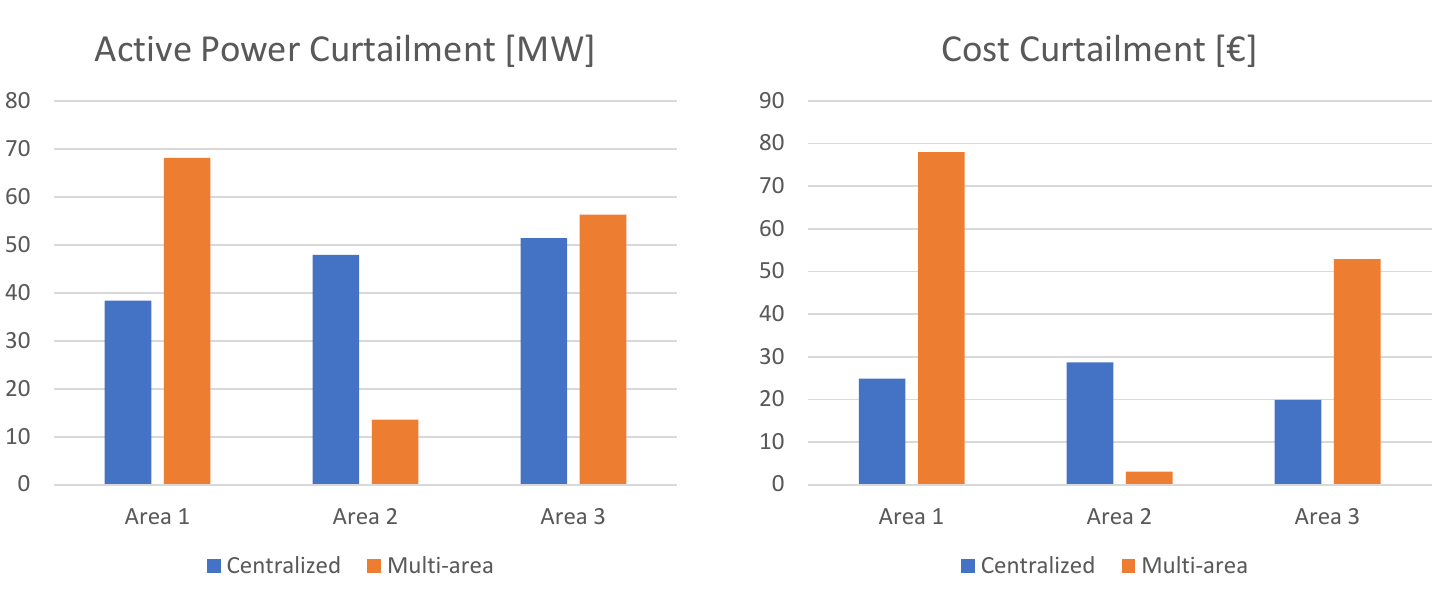}
    \caption{Active power curtailment (left) and correspondent curtailment cost (right), namely, the purely monetary part of the cost functions $J_i$, for the different areas in Fig.~\ref{fig:IEEE30bus} at the steady-state operation obtained by the multi-area OFO controllers in \eqref{eq:OFOsinglearea} and the centralized OFO controller.}
    \label{fig:tables_cropped}
\end{figure}


\section{Conclusion and Outlook}
Online Feedback Optimization is an attracting control methodology for real-time operation of large scale transmission grids, as it is requires minimal information on the grid model and it is robust to exogenous disturbances affecting the grid (e.g., uncontrollable loads and variable generation). In multi-area transmission grids, where the sub-areas autonomously regulate their congestion via local OFO controllers, the resulting closed-loop dynamics can be studied using tools form game theory and monotone operator theory. We showed that, under some technical conditions on the parameters of the local OFO controllers, the multi-area transmission grid stabilizes to a competitive equilibrium. We show via numerical simulations that such equilibria are globally inefficient, compared to a centralized socially-optimal steady state. On the other hand, socially-optimal steady-state are not strategically stable nor incentive compatible, meaning that some areas may be heavily penalized by participating in a centralized control design.
Hierarchical game theory offers a solution to this inherent trade off in the form of incentives that need to be properly designed in order to make the socially optimal state strategically stable. The design of these incentives for real-time operation of modern multi-area transmission grids is largely unexplored.


\appendix
\section*{Proof of Theorem \ref{th:main}}
We start by proving existence of A Nash equilibrium of \eqref{eq:game}.
By \cite[Corollary~3.4]{facchinei2010generalized}, the Nash equilibria of \eqref{eq:game} correspond to the solution of the following generalized equation (GE), or variational inequality, $0 \in \mathcal N_{\mathcal U}(u) + \mathbb{F}(u)$, where $\mathcal N_{\mathcal U}$ is the normal cone operator of the Cartesian set ${\mathcal U}=\prod_{i \in \mathcal{I}}\mathcal{U}_i$. Moreover, by \cite[Proposition~23.36]{bauschke2011convex}, this GE admits at least one solution since the $\mathcal U_i$ are bounded and $\mathbb{F}$ is cocoercive (hence, also maximally monotone\cite[Example 20.31]{bauschke2011convex}) by Assumption~\ref{ass:PGmon}. It follows that at least a Nash equilibria of \eqref{eq:game} exists.
Next, we show that the closed-loop discrete-time dynamics in \eqref{eq:OFOsinglearea} globally converge to a Nash equilibrium. First, we note that by stacking up \eqref{eq:OFOsinglearea}, we can re-write them more compactly as 
\begin{align}
\label{eq:FB}
u^{k+1} &= \textrm{proj}_{\mathcal{U}}\left(u - \Gamma \mathbb F(u) \right)\\
&= (\Id + \Gamma \mathcal N_{\mathcal U})^{-1} \left(u - \Gamma \mathbb F(u) \right) ,
\end{align}
where $\Gamma = \textrm{diag}(\gamma_1,\ldots,\gamma_N)$ is a diagonal matrix with the control gains $\gamma_i$ on the main diagonal. The iteration \eqref{eq:FB} corresponds to a forward-backward splitting algorithm \cite[\S~26.5]{bauschke2011convex}, which converges to a zero of $\Gamma \mathbb F + \Gamma \mathcal N_{\mathcal U}$ under the conditions of \cite[Theorem~26.14]{bauschke2011convex}. Namely, cocoercivity of $\Gamma \mathbb F$, maximal monotonicity of $\Gamma \mathcal N_{\mathcal U}$, and existence of a solution. The first condition holds in the weighted norm $\|\cdot\|_{\Gamma^{-1}}$, with parameter $1/2$, since for all $u,u' \in \mathcal{U}$
\begin{align*}
\langle \mathbb F(u) - \mathbb F(u'), \; u-u'\rangle_{\Gamma^{-1}} & \geq  \mu \|\mathbb F(u) - \mathbb F(u')\|^{2},\\
& \geq \displaystyle \mu/ (\max_{i\in\mathcal{I}} \gamma_i) \, \|\mathbb F(u) - \mathbb F(u')\|_{\Gamma^{-1}}^{2},\\
&> 1/2 \, \|\mathbb F(u) - \mathbb F(u')\|_{\Gamma^{-1}}^{2},
\end{align*}
where the first inequality follows from the cocoercivity of $\mathbb F$ (Assumption \eqref{ass:PGmon}), the second since $\|\cdot\|^2_{\Gamma^{-1}} \geq \eigmin{\Gamma^{-1}} \| \cdot \|^2 $ with $\eigmin{\Gamma^{-1}} = (\max_{i\in\mathcal{I}} \gamma_i)^{-1}$, and the third since $\gamma_i \in (0,2\mu)$ by Assumption~\ref{ass:StepSize}.
The second condition, namely maximal monotonicity of $\Gamma \mathcal N_{\mathcal U}$, holds in the same norm, $\|\cdot\|_{\Gamma^{-1}}$ by \cite[Example~20.26, Proposition~20.24]{bauschke2011convex} since $\Gamma$ is a diagonal positive  matrix. Finally, we note that $0 \in \Gamma (\mathbb F(u) + \mathcal N_{\mathcal U}(u))$ if and only if $ 0 \in  \mathbb F(u) + \mathcal N_{\mathcal U}(u))$. Note that the latter GE characterizes the set of Nash equilibria of \eqref{eq:game}, which is nonempty by the first part of the proof. Now, we can finally invoke \cite[Theorem~26.14]{bauschke2011convex} to prove convergence of the sequence  $\{u^k \}_{k \in \mathbb N}$ generated by \eqref{eq:FB} to some $u^*$ such that $0 \in \textrm{zer}( \mathbb F(u^*) + \mathcal N_{\mathcal U}(u^*))$, which is a Nash equilibrium of \eqref{eq:game}, and concludes the proof.
{\hfill $\square$}

\balance
\bibliographystyle{IEEEtran}
\bibliography{SaverioBolognani20230915.bib,biblio.bib} 

\end{document}